\documentclass[12pt,reqno]{amsart}
\usepackage{etoolbox}

   \makeatletter

 \patchcmd{\@setaddresses}{\scshape\ignorespaces}{\ignorespaces}{}{} 

\appto\maketitle{%
\let\@makefnmark\relax  \let\@thefnmark\relax
\ifx\@empty\addresses\else\@footnotetext{%
  \vskip-\bigskipamount\@setaddresses}
  }
\def\enddoc@text{}
\makeatother

\makeatletter
\patchcmd\maketitle
  {\uppercasenonmath\shorttitle}
  {}
  {}{}
\patchcmd\maketitle
  {\@nx\MakeUppercase{\the\toks@}}
  {\the\toks@}
  {}
  {}{}
\patchcmd\@settitle{\uppercasenonmath\@title}{\Large}{}{}
\patchcmd\@setauthors
  {\MakeUppercase{\authors}}
  {\authors}
  {}{}
\makeatother
\usepackage{amsmath,amssymb,amsthm}
\usepackage{color}
\usepackage{url}
\usepackage{tikz-cd}
\usepackage[utf8]{inputenc}
\usepackage[T1]{fontenc}
\newtheorem{theorem}{Theorem}[section]
\newtheorem{definition}{Definition}[section]

\newtheorem{corollary}{Corollary}[section]
\newtheorem{proposition}{Proposition}[section]
\newtheorem{lemma}{Lemma}[section]
\newtheorem{remark}{Remark}[section]

\numberwithin{equation}{section}
\usepackage[colorlinks=true]{hyperref}
\hypersetup{urlcolor=blue, citecolor=red , linkcolor= blue}
\usepackage[capitalise,noabbrev,nameinlink]{cleveref}      
\begin{document}
\address{$^{[1]}$ University of Sfax, Sfax, Tunisia.}
\email{\url{kais.feki@hotmail.com}}


\date{\today}
\author[Kais Feki] {\Large{Kais Feki}$^{1}$}
\title[Some numerical radius inequalities for semi-Hilbert space operators]{Some numerical radius inequalities for semi-Hilbert space operators}

\maketitle

\begin{abstract}
Let $A$ be a positive bounded linear operator acting on a complex Hilbert space $\big(\mathcal{H}, \langle \cdot\mid \cdot\rangle \big)$. Let $\omega_A(T)$ and ${\|T\|}_A$ denote the $A$-numerical radius and the $A$-operator seminorm of an operator $T$ acting on the semi-Hilbertian space
$\big(\mathcal{H}, {\langle \cdot\mid \cdot\rangle}_A\big)$ respectively, where ${\langle x\mid y\rangle}_A := \langle Ax\mid y\rangle$ for all $x, y\in\mathcal{H}$. In this paper, we show that
  \begin{equation*}\label{m1}
  \tfrac{1}{4}\|T^{\sharp_A} T+TT^{\sharp_A}\|_A\le  \omega_A^2\left(T\right) \le \tfrac{1}{2}\|T^{\sharp_A} T+TT^{\sharp_A}\|_A.
  \end{equation*}
Here $T^{\sharp_A}$ is denoted to be a distinguished $A$-adjoint operator of $T$. Moreover, a considerable improvement of the above inequalities is proved. This allows to compute the $\mathbb{A}$-numerical radius of the operator matrix $\begin{pmatrix}
I&T \\
0&-I
\end{pmatrix}$ where $\mathbb{A}= \text{diag}(A,A)$. In addition, several $A$-numerical radius inequalities for semi-Hilbertian space operators are also established.
\end{abstract}

\section{Introduction and Preliminaries}\label{s1}
In this paper, let $\big(\mathcal{H}, \langle\cdot\mid \cdot\rangle\big)$ be a complex Hilbert space equipped with the norm $\|\cdot\|$. Let $\mathcal{B}(\mathcal{H})$ stand for the $C^*$-algebra of all bounded linear operators defined on $\mathcal{H}$ with the identity operator $I$. An operator $T\in\mathcal{B}\left( \mathcal{H}\right)$ is said to be positive if $\langle Tx\mid x\rangle \geq 0$ for all $x\in\mathcal{H}$. The square root of a positive operator $T$ is denoted by $T^{1/2}$. Let $|T|$ denotes the square root of $T^*T$, where $T^*$ is the adjoint of $T$. The operator norm and the spectral radius of an operator $T$ are denoted by $\left\| T \right\|$ and $r\left( T \right)$ respectively and they are given by
\begin{equation*}
\left\| T \right\|=\sup \left\{ \left\| Tx \right\|\,;\text{ }x\in \mathcal{H},\left\| x \right\|=1 \right\}.
\end{equation*}
and
\begin{equation*}
r(T)=\sup \left\{ |\lambda|\,; \lambda \in \sigma(T)\right \},
\end{equation*}
where $\sigma(T)$ denotes the spectrum  of $T$. It is well-known that, for every $T\in\mathcal{B}(\mathcal{H})$, we have
\begin{equation}\label{absolute}
\left\| T \right\|^2=\left\|\,| T |^2\,\right\|=\left\|\,| T^* |^2\,\right\|=r(| T |^2)=r(| T^* |^2).
\end{equation}

From now on, $A$ always stands for a positive linear operator in $\mathcal{B}(\mathcal{H})$. The cone of all positive operators of $\mathcal{B}(\mathcal{H})$ will be denoted by $\mathcal{B}(\mathcal{H})^+$. In this article, for a given $A\in\mathcal{B}(\mathcal{H})^+$, we are going to consider an additional semi-inner product ${\langle \cdot\mid \cdot\rangle}_A$ on $\mathcal{H}$ defined by
$${\langle x\mid y\rangle}_A = \langle Ax\mid y\rangle,\;\;\forall\,x, y \in\mathcal{H}.$$
This makes $\mathcal{H}$ into a semi-Hilbertian space. The seminorm induced by ${\langle \cdot\mid \cdot\rangle}_A$ is given by ${\|x\|}_A=\sqrt{{\langle x\mid x\rangle}_A}$ for every $x\in\mathcal{H}$. It can be observed that $(\mathcal{H}, {\|\cdot\|}_A)$ is a normed space if and only if $A$ is one-to-one and that the semi-Hilbertian space $(\mathcal{H}, {\|\cdot\|}_A)$ is complete if and only if the range of $A$ is closed in $\mathcal{H}$. In what follows, by an operator we mean a bounded linear operator in $\mathcal{B}(\mathcal{H})$. Also, the range of every operator $T$ is denoted by $\mathcal{R}(T)$, its null space by $\mathcal{N}(T)$. Given $\mathcal{M}$ a linear subspace of $\mathcal{H}$, $\overline{\mathcal{M}}$ denotes the closure with respect to the norm topology of $\mathcal{H}$. If $\mathcal{M}$ is a closed subspace of $\mathcal{H}$, then $P_{\mathcal{M}}$ stands for the orthogonal projection onto $\mathcal{M}$.

The semi-inner product ${\langle \cdot\mid \cdot\rangle}_{A}$ induces on the quotient $\mathcal{H}/\mathcal{N}(A)$ an inner product which is not complete unless $\mathcal{R}(A)$ is closed. However, a canonical construction due to de Branges and Rovnyak \cite{branrov} shows that the completion of $\mathcal{H}/\mathcal{N}(A)$ is isometrically isomorphic to the Hilbert space $\mathcal{R}(A^{1/2})$ equipped with the following inner product
\begin{align*}
\langle A^{1/2}x,A^{1/2}y\rangle_{\mathbf{R}(A^{1/2})}:=\langle P_{\overline{\mathcal{R}(A)}}x\mid P_{\overline{\mathcal{R}(A)}}y\rangle,\;\forall\, x,y \in \mathcal{H}.
\end{align*}
 For the sequel, the Hilbert space $\big(\mathcal{R}(A^{1/2}), \langle\cdot,\cdot\rangle_{\mathbf{R}(A^{1/2})}\big)$ will be denoted by $\mathbf{R}(A^{1/2})$.

\begin{definition} (\cite{acg1})
Let $T \in \mathcal{B}(\mathcal{H})$. An operator $S\in\mathcal{B}(\mathcal{H})$ is called an $A$-adjoint of $T$ if for every $x,y\in \mathcal{H}$, the identity
$\langle Tx\mid y\rangle_A=\langle x\mid Sy\rangle_A$ holds. That is $AS=T^*A$.
\end{definition}
Generally, the existence of an $A$-adjoint operator is not guaranteed. The set of all operators in $\mathcal{B}(\mathcal{H})$ admitting $A$-adjoints is denoted by $\mathcal{B}_{A}(\mathcal{H})$. By Douglas Theorem \cite{doug}, we  have
$$\mathcal{B}_{A}(\mathcal{H})=\left\{T\in \mathcal{B}(\mathcal{H})\,;\;\mathcal{R}(T^{*}A)\subseteq \mathcal{R}(A)\right\}.$$
Further, the set of all operators admitting $A^{1/2}$-adjoints is denoted by $\mathcal{B}_{A^{1/2}}(\mathcal{H})$. Again, by applying Douglas Theorem, we obtain
$$\mathcal{B}_{A^{1/2}}(\mathcal{H})=\left\{T \in \mathcal{B}(\mathcal{H})\,;\;\exists \,\lambda > 0\,;\;\|Tx\|_{A} \leq \lambda \|x\|_{A},\;\forall\,x\in \mathcal{H}  \right\}.$$
If $T\in \mathcal{B}_{A^{1/2}}(\mathcal{H})$, we will say that $T$ is $A$-bounded. Notice that $\mathcal{B}_{A}(\mathcal{H})$ and $\mathcal{B}_{A^{1/2}}(\mathcal{H})$ are two subalgebras of $\mathcal{B}(\mathcal{H})$ which are, in general, neither closed nor dense in $\mathcal{B}(\mathcal{H})$. Moreover, the following inclusions $\mathcal{B}_{A}(\mathcal{H})\subseteq \mathcal{B}_{A^{1/2}}(\mathcal{H})\subseteq \mathcal{B}(\mathcal{H})$ hold with equality if $A$ is injective and has closed range. For an account of results, we refer to \cite{acg1,acg2,bakfeki01,feki01} and the references therein. Clearly, $\langle\cdot\mid\cdot\rangle_{A}$ induces a seminorm on $\mathcal{B}_{A^{1/2}}(\mathcal{H})$. Indeed, if $T\in\mathcal{B}_{A^{1/2}}(\mathcal{H})$, then
\begin{equation*}\label{semii}
\|T\|_A:=\sup_{\substack{x\in \overline{\mathcal{R}(A)},\\ x\not=0}}\frac{\|Tx\|_A}{\|x\|_A}=\sup\big\{{\|Tx\|}_A\,; \,\,x\in \mathcal{H},\, {\|x\|}_A =1\big\}<\infty.
\end{equation*}
Notice that, if $T\in\mathcal{B}_{A^{1/2}}(\mathcal{H})$, then $\|T\|_A=0$ if and only if $AT=0$. Further, it was proved in \cite{fg} that for $T\in\mathcal{B}_{A^{1/2}}(\mathcal{H})$ we have
\begin{equation}\label{newsemi}
\|T\|_A=\sup\left\{|\langle Tx\mid y\rangle_A|\,;\;x,y\in \mathcal{H},\,\|x\|_{A}=\|y\|_{A}= 1\right\}.
\end{equation}
It should be emphasized that it may happen that ${\|T\|}_A = + \infty$ for some $T\in\mathcal{B}(\mathcal{H})\setminus \mathcal{B}_{A^{1/2}}(\mathcal{H})$ (see \cite[Example 2]{feki01}).

Before we move on, it should be mentioned that for $T\in \mathcal{B}_{A^{1/2}}(\mathcal{H})$ we have
\begin{equation}\label{semiiineq}
\|Tx\|_A\leq \|T\|_A\|x\|_A,\;\forall\,x\in \mathcal{H}.
\end{equation}
Also, we would like to emphasize that \eqref{semiiineq} fails to hold in general for some $T\in\mathcal{B}(\mathcal{H})$. In fact, one can take the operators
$A= \begin{pmatrix}0 & 0 \\ 0 & 1\end{pmatrix}$ and $T= \begin{pmatrix} 0 & 1 \\ 1 & 0 \end{pmatrix}$ on $\mathbb{C}^2$. If $x=(1,0)$, then $\|x\|_A=0$ and $\|Tx\|_A=1$. Thus, \eqref{semiiineq} fails to be true. Moreover, by applying \eqref{semiiineq} we show that
\begin{equation}\label{sousmultiplicative}
\|TS\|_A\leq \|T\|_A\|S\|_A,
\end{equation}
 for every $T,S\in \mathcal{B}_{A^{1/2}}(\mathcal{H})$.

If $T\in\mathcal{B}_{A}(\mathcal{H})$, then by Douglas theorem there exists a unique $A$-adjoint of $T$, denoted by $T^{\sharp_A}$, which satisfies $\mathcal{R}(T^{\sharp_A})\subseteq \overline{\mathcal{R}(A)}$. We observe that $T^{\sharp_A}=A^\dag T^*A,$ where $A^\dag$ is the Moore-Penrose inverse of $A$. For results concerning $T^{\sharp_A}$ and $A^\dag$ see \cite{acg1,acg2}. Notice that if $T\in\mathcal{B}_{A}(\mathcal{H})$, then $T^{\sharp_A}\in\mathcal{B}_{A}(\mathcal{H})$ and $(T^{\sharp_A})^{\sharp_A} = P_{\overline{\mathcal{R}(A)}}TP_{\overline{\mathcal{R}(A)}}$. For more facts related to this class of operators, we invite the reader to \cite{acg1,acg2} and their references. Now, we recall that an operator $T\in \mathcal{B}(\mathcal{H})$ is said to be $A$-isometry if $\|Tx\|_A=\|x\|_A$ for all $x\in \mathcal{H}$. Further, an operator $U\in \mathcal{B}_A(\mathcal{H})$ is called $A$-unitary if $U$ and $U^{\sharp_A}$ are $A$-isometries. By an $A$-normal operator, we mean an operator $T\in \mathcal{B}_A(\mathcal{H})$ which satisfies $T^{\sharp_A}T=TT^{\sharp_A}$. For more details related to these classes of operators, the reader can consult \cite{acg1,bakfeki04}. Any operator $T\in {\mathcal B}_A({\mathcal H})$ can be represented as $T=\Re_A(T)+i\Im_A(T)$, where
$$\Re_A(T):=\frac{T+T^{\sharp_A}}{2}\;\;\text{ and }\;\;\Im_A(T):=\frac{T-T^{\sharp_A}}{2i}.$$

Recently, the $A$-numerical range of an operator $T\in\mathcal{B}(\mathcal{H})$ is defined by Baklouti et al. in \cite{bakfeki01} as
$W_A(T) = \big\{{\langle Tx\mid x\rangle}_A\,; \,\, x\in \mathcal{H},\, {\|x\|}_A = 1\big\}$.
It was shown in \cite{bakfeki01} that $W_A(T)$ is a nonempty convex subset of $\mathbb{C}$ which is not necessarily closed even if $\text{dim}(\mathcal{H})<\infty$. Notice that supremum modulus of $W_A(T)$ is called
the $A$-numerical radius of $T$ (see \cite{bakfeki01}). More precisely, we have
$$\omega_A(T) = \sup\big\{|\lambda|\,; \,\, \lambda\in W_A(T)\big\}=\sup\left\{\big|{\langle Tx\mid x\rangle}_A\big|\,; \,\,\, x\in \mathcal{H},\, {\|x\|}_A = 1\right\}.$$
It should be mentioned that $W_A(T)= \mathbb{C}$ when $T\in\mathcal{B}(\mathcal{H})$ and satisfies $T(\mathcal{N}(A))\not\subseteq\mathcal{N}(A)$ (\cite[Theorem 2.1.]{bakfeki01}). So, $\omega_A(T) = + \infty$ for every operator $T\in\mathcal{B}(\mathcal{H})$ such that $T(\mathcal{N}(A))\not\subseteq\mathcal{N}(A)$. Notice that $\omega_A(\cdot)$ is a seminorm on $\mathcal{B}_{A^{1/2}}(\mathcal{H})$ with is equivalent to the $A$-operator seminorm. More precisely, it was shown in \cite{bakfeki01} that for every $T\in \mathcal{B}_{A^{1/2}}(\mathcal{H})$, we have
\begin{equation}\label{refine1}
\tfrac{1}{2} \|T\|_A\leq\omega_A(T) \leq \|T\|_A.
\end{equation}

By an $A$-selfadjoint, we mean an $T\in\mathcal{B}(\mathcal{H})$ which satisfies $AT$ is selfadjoint, that is, $AT = T^*A$. Further, for every $A$-selfadjoint operator $T$ we have
\begin{equation}\label{aself1}
\|T\|_{A}=\omega_A(T):=\sup\left\{|\langle Tx\mid x\rangle_{A}|\,;\;x\in \mathcal{H},\;\|x\|_{A}= 1\right\}.
\end{equation}
(see \cite{feki01}). In addition, an operator $T$ is said to be $A$-positive if $AT\geq0$ and we write $T\geq_{A}0$. One can verify that $T^{\sharp_A} T\geq_{A}0$ and $TT^{\sharp_A}\geq_{A}0$. Moreover, in view of \cite[Proposition 2.3.]{acg2}, we have
\begin{align}\label{diez}
{\|T^{\sharp_A} T\|}_A = {\|TT^{\sharp_A}\|}_A = {\|T\|}^2_A = {\|T^{\sharp_A}\|}^2_A.
\end{align}
For a given operator $T$, the $A$-Crawford number of $T$ is defined, as in \cite{zamani1}, by
\begin{align*}
c_A(T) = \inf \big\{|{\langle Tx\mid x\rangle}_A|\,;\,\, x\in\mathcal{H},\,{\|x\|}_A =1\big\}.
\end{align*}
Recently, the present author proved in \cite{feki01} some $A$-numerical radius inequalities for $A$-bounded operators. In particular, he showed that for every $T\in\mathcal{B}_{A^{1/2}}(\mathcal{H})$ and all positive integer $n$, we have
\begin{equation}\label{apower}
\omega_A(T^n)\leq [\omega_A(T)]^n.
\end{equation}
 Also, it has been shown in \cite{feki01} that for every $T\in\mathcal{B}_{A^{1/2}}(\mathcal{H})$, we have
\begin{equation}\label{kaisnew01}
\omega_A(T)\leq \frac{1}{2}\left(\|T\|_A+\|T^2\|_A^{1/2}\right).
\end{equation}
Clearly, \eqref{kaisnew01} is a refinement of the second inequality in \eqref{refine1}. For other facts and results related to the concept of $A$-numerical radius, the reader is referred to \cite{bakfeki01,bakfeki04,feki01,zamani1} and the references therein. In recent years, several results covering some classes of operators on a complex Hilbert space $\big(\mathcal{H}, \langle \cdot\mid \cdot\rangle\big)$ were extended to $\big(\mathcal{H}, {\langle \cdot\mid \cdot\rangle}_A\big)$. The reader is invited to see \cite{bakfeki01,bakfeki04,bhunfekipaul,feki03,zamani2,tamzhang,zamani1,zamani3} and the references therein. In this article, we will establish several results governing $\omega_A(\cdot)$ and $\|\cdot\|_A$. Some of these results will be a natural extensions of the well-known case $A=I$ due to Kittaneh et al. \cite{A.K.1, OK2,FK,ak2020}.

 \section{Results}\label{s2}
In this section, we present our results. In order to prove our first result, we need the following lemmas.
\begin{lemma}\label{lm1}(\cite[Theorem 2.6.]{zamani1})
Let $T\in\mathcal{B}_{A}(\mathcal{H})$. Then
\begin{align*}
\omega_A(T) = \displaystyle\sup\Big\{\left\|\alpha \Re_A(T) +
\beta \Im_A(T)\right\|_A\,;\;\alpha, \beta \in \mathbb{R}\; ,\,\alpha^2 + \beta^2 = 1\Big\}.
\end{align*}
\end{lemma}

\begin{lemma}\label{lm2}(\cite[Theorem 2.5.]{zamani1})
Let $T\in\mathcal{B}_{A}(\mathcal{H})$. Then
\begin{align*}
\omega_A(T) = \displaystyle{\sup_{\theta \in \mathbb{R}}}{\left\|\Re_A(e^{i\theta}T)\right\|}_A=\displaystyle{\sup_{\theta \in \mathbb{R}}}{\left\|\Im_A(e^{i\theta}T)\right\|}_A
\end{align*}
\end{lemma}

\begin{lemma}\label{lm3}(\cite[Lemma 1]{feki03})
Let $T\in \mathcal{B}(\mathcal{H})$ be an $A$-selfadjoint operator. Then, $T^{\sharp_A}$ is $A$-selfadjoint and
\begin{equation*}
({T^{\sharp_A}})^{\sharp_A}=T^{\sharp_A}.
\end{equation*}
\end{lemma}

\begin{lemma}\label{lm4}(\cite[Lemma 6]{bhunfekipaul})
Let $\mathbb{A}=\begin{pmatrix}
A&0\\
0&A
\end{pmatrix}$ and $\mathbb{T}=\begin{pmatrix}
T_{11}&T_{12} \\
T_{21}&T_{22}
\end{pmatrix}$ be such that $T_{ij}\in \mathcal{B}_{A}(\mathcal{H})$ for all $i,j\in\{1,2\}$. Then, $\mathbb{T}\in \mathcal{B}_{\mathbb{A}}(\mathcal{H}\oplus \mathcal{H})$ and
$$ \mathbb{T}^{\sharp_\mathbb{A}}=\begin{pmatrix}
T^{\sharp_A}_{11} & T^{\sharp_A}_{21} \\
T^{\sharp_A}_{12} & T^{\sharp_A}_{22}
\end{pmatrix}.$$
\end{lemma}

Now, we are in a position to prove our first result in this paper.
\begin{theorem}\label{main1}
Let $T\in \mathcal{B}_{A}(\mathcal{H})$. Then,
  \begin{equation}\label{feki1}
  \tfrac{1}{4}\|T^{\sharp_A} T+TT^{\sharp_A}\|_A\le  \omega_A^2(T) \le \tfrac{1}{2}\|T^{\sharp_A} T+TT^{\sharp_A}\|_A.
  \end{equation}
Moreover, the inequalities in \eqref{feki1} are sharp.
\end{theorem}
\begin{proof}
By using the property $AT^{\sharp_A}=T^*A$ we see that
$$A[\Re_A(T)]=\tfrac{1}{2}AT+\tfrac{1}{2}AT^{\sharp_A}=\tfrac{1}{2}(T^{\sharp_A})^*A+\tfrac{1}{2}T^*A=[\Re_A(T)]^*A.$$
Hence, $\Re_A(T)$ is an $A$-selfadjoint operator. Similarly, we prove that $\Im_A(T)$ is $A$-selfadjoint. So, by Lemma \ref{lm3} we have
\begin{equation}
({[\Re_A(T)]^{\sharp_A}})^{\sharp_A}=[\Re_A(T)]^{\sharp_A}\;\text{ and }\; ({[\Im_A(T)]^{\sharp_A}})^{\sharp_A}=[\Im_A(T)]^{\sharp_A}.
\end{equation}

Now, it can be observed that
\begin{align}\label{newway00}
\left([\Re_A(T)]^{\sharp_A}\right)^2 + \left([\Im_A(T)]^{\sharp_A}\right)^2
&=\frac{1}{2}\left[(T^{\sharp_A})^{\sharp_A} T^{\sharp_A}+ T^{\sharp_A}(T^{\sharp_A})^{\sharp_A}\right]\nonumber\\
&= \frac{1}{2}\left(TT^{\sharp_A} + T^{\sharp_A} T\right)^{\sharp_A}.
\end{align}
Let $\mathbb{A}=\begin{pmatrix}
A &0\\
0 &A
\end{pmatrix}$ and $x\in\mathcal{H}$ be such that $\|x\|_A=1$. Let also $\alpha,\beta \in \mathbb{R}$ be such that $\alpha ^{2}+\beta
^{2}=1$. It can seen that
\begin{align*}
&\left\Vert\left(\alpha [\Re_A(T)]^{\sharp_A}+\beta [\Im_A(T)]^{\sharp_A}\right)x\right\Vert_A^2\\
& =\left\Vert
\begin{pmatrix}
[\Re_A(T)]^{\sharp_A} & [\Im_A(T)]^{\sharp_A} \\
0 & 0%
\end{pmatrix}
\begin{pmatrix}
\alpha x \\
\beta x%
\end{pmatrix}
\right\Vert_A^2 \\
&\leq
\left\Vert
\begin{pmatrix}
[\Re_A(T)]^{\sharp_A} & [\Im_A(T)]^{\sharp_A} \\
0 & 0%
\end{pmatrix}
\right\Vert_{\mathbb{A}}^2 \quad (\text{by }\,\eqref{semiiineq})\\
&=
\left\Vert
\begin{pmatrix}
[\Re_A(T)]^{\sharp_A} & [\Im_A(T)]^{\sharp_A} \\
0 & 0%
\end{pmatrix}
\begin{pmatrix}
[\Re_A(T)]^{\sharp_A} & [\Im_A(T)]^{\sharp_A} \\
0 & 0%
\end{pmatrix}^{\sharp_\mathbb{A}}
\right\Vert_{\mathbb{A}}\quad (\text{by }\,\eqref{diez})\\
&=
\left\Vert
\begin{pmatrix}
[\Re_A(T)]^{\sharp_A} & [\Im_A(T)]^{\sharp_A} \\
0 & 0%
\end{pmatrix}
\begin{pmatrix}
[\Re_A(T)]^{\sharp_A} & 0 \\
[\Im_A(T)]^{\sharp_A} & 0%
\end{pmatrix}
\right\Vert_{\mathbb{A}}\\
& =\left\Vert ([\Re_A(T)]^{\sharp_A})^{2}+([\Im_A(T)]^{\sharp_A})^{2}\right\Vert_A\\
&=\frac{1}{2}\left\Vert \left(TT^{\sharp_A} + T^{\sharp_A} T\right)^{\sharp_A}\right\Vert_A\quad (\text{by }\,\eqref{newway00})\\
&=\frac{1}{2}\left\Vert TT^{\sharp_A} + T^{\sharp_A} T\right\Vert_A\quad (\text{by }\,\eqref{diez}).
\end{align*}%
Therefore, we obtain
\begin{align}\label{qq4}
\omega_A^{2}(T)
& = \omega_A^{2}(T^{\sharp_A})\nonumber\\
 &=\sup_{\alpha ^{2}+\beta ^{2}=1}\left\Vert\alpha [\Re_A(T)]^{\sharp_A}+\beta [\Im_A(T)]^{\sharp_A}\right\Vert_A^2\nonumber\quad(\text{by Lemma } \ref{lm2})\\
 &\leq \frac{1}{2}\left\Vert TT^{\sharp_A} + T^{\sharp_A} T\right\Vert_A.
\end{align}

 Let $\theta \in \mathbb{R}$. By making simple computations we see that
\begin{align}\label{newway}
\left[\Re_A(e^{i\theta}T)\right]^2 + \left[\Im_A(e^{i\theta}T)\right]^2
= \frac{1}{2}\left(TT^{\sharp_A} + T^{\sharp_A} T\right).
\end{align}
This implies, that
\begin{align*}
\frac{1}{2}\left\|\left(TT^{\sharp_A} + T^{\sharp_A} T\right)\right\|_A
& =\left\|\left[\Re_A(e^{i\theta}T)\right]^2 + \left[\Im_A(e^{i\theta}T)\right]^2\right\|_A \\
& \leq\left\|\left[\Re_A(e^{i\theta}T)\right]^2 \right\|_A+ \left\|\left[\Im_A(e^{i\theta}T)\right]^2\right\|_A \\
 &\leq 2\omega_A^2(T).\;\;(\text{by }\eqref{sousmultiplicative}\text{ and }\text{Lemma } \ref{lm2})
\end{align*}
Hence, we get
\begin{equation}\label{n1}
\frac{1}{4}\left\|\left(TT^{\sharp_A} + T^{\sharp_A} T\right)\right\|_A\leq \omega_A^2(T).
\end{equation}
Combining \eqref{qq4} together with \eqref{n1} yields to \eqref{feki1} as desired. Now, to see that the constant $\frac{1}{2}$ is sharp we consider an arbitrary $A$-normal operator $T$. By \cite[Theorem 4]{feki01} we have $\|T\|_A=\omega_A(T)$. So, we get
$$\tfrac{1}{2}\|T^{\sharp_A} T+TT^{\sharp_A}\|_A=\|TT^{\sharp_A}\|_A=\|T\|_A^2=\omega_A^2(T).$$
The sharpness of the first inequality in \eqref{feki1} can be verified by considering the operators $A=\begin{pmatrix}\alpha&0\\0&\alpha\end{pmatrix}$ for some $\alpha\neq0$ and $T=\begin{pmatrix}0&1\\0&0\end{pmatrix}$
\end{proof}
\begin{remark}
\noindent (1) If $A=I$, we obtain the well-known inequalities proved by F. Kittaneh in \cite[Theorem 1]{FK}. Also, if $A$ is an injective positive operator, we get the recent result proved by Bhunia et al. proved in \cite[Corollary 2.7]{BPN}.
\par \vskip 0.1 cm \noindent (2)\;The second inequality in Theorem \ref{main1} has recently been proved by Zamani in \cite[Theorem 2.10]{zamani1} using a completely different argument. Moreover, it is not difficult to see that
\begin{align*}
\frac{1}{2}{\|T\|}_A\leq\frac{1}{2}\sqrt{{\big\|T T^{\sharp_A} + T^{\sharp_A} T\big\|}_A}\leq\omega_A(T) \leq \frac{\sqrt{2}}{2}\sqrt{{\big\|T T^{\sharp_A} + T^{\sharp_A} T\big\|}_A} \leq {\|T\|}_A.
\end{align*}
So, the inequalities in Theorem \ref{main1} improve the inequalities in \eqref{refine1}.
\end{remark}

In order to prove our next result in this paper, we need the following lemma.

\begin{lemma}\label{jdidddd}(\cite{feki04})
Let $T,S\in \mathcal{B}_{A}(\mathcal{H})$ and $\mathbb{A}=\begin{pmatrix}
A &0\\
0 &A
\end{pmatrix}$. Then,
\begin{equation*}
\omega_{\mathbb{A}}\left[\begin{pmatrix}
0&T\\
S &0
\end{pmatrix}\right]=\frac{1}{2}\sup_{\theta\in \mathbb{R}}\left\|e^{i\theta}T+e^{-i\theta}S^{\sharp_A}\right\|_A.
\end{equation*}
In particular,
\begin{equation}\label{ppp}
\omega_{\mathbb{A}}\left[\begin{pmatrix}
0&T\\
T &0
\end{pmatrix}\right]=\omega_A(T).
\end{equation}
\end{lemma}

Now, we are in a position to prove our second result in this paper.
 \begin{theorem}
Let $\mathbb{A}=\begin{pmatrix}A&0\\0&A\end{pmatrix}$ and $T,S\in \mathcal{B}_{A}(\mathcal{H})$. Then,
\begin{equation}\label{r62}
\omega_{\mathbb{A}}\left[\begin{pmatrix}
0 &T\\
S &0
\end{pmatrix}\right]\leq \min\big\{\Psi_A(T,S),\Psi_A(S,T)\big\}\leq\frac{\|T+S\|_A}{2},
\end{equation}
and
\begin{equation}\label{r61}
\omega_{\mathbb{A}}\left[\begin{pmatrix}
0 &T\\
S &0
\end{pmatrix}\right]\geq\max\big\{\Phi_A(T,S),\Phi_A(S,T)\big\},
\end{equation}
where
$$\Psi_A(T,S)=\frac{1}{2}\sqrt{{\big\|TT^{\sharp_A} + S^{\sharp_A}S\big\|}_A+ 2\omega_A(TS)},$$
and
$$\Phi_A(T,S)=\frac{1}{2}\sqrt{{\big\|TT^{\sharp_A} + S^{\sharp_A}S\big\|}_A + 2c_A(TS)}.$$
\end{theorem}
\begin{proof}
By Lemma \ref{jdidddd}, we see that
\begin{align*}
\omega_{\mathbb{A}}\left[\begin{pmatrix}
0&T\\
S &0
\end{pmatrix}\right]
&=\omega_{\mathbb{A}}\left[\begin{pmatrix}
0&T\\
S &0
\end{pmatrix}^{\sharp_\mathbb{A}}\right]\\
&=\frac{1}{2}\sup_{\theta\in \mathbb{R}}\left\|e^{i\theta}S^{\sharp_A}+e^{-i\theta}(T^{\sharp_A})^{\sharp_A}\right\|_A\\
& = \frac{1}{2}\displaystyle{\sup_{\theta \in \mathbb{R}}}
{\left\|\Big(e^{i\theta}S^{\sharp_A}+e^{-i\theta}(T^{\sharp_A})^{\sharp_A}\Big)
\Big(e^{i\theta}S^{\sharp_A}+e^{-i\theta}(T^{\sharp_A})^{\sharp_A}\Big)^{\sharp_A}\right\|}_A^{\frac{1}{2}}\,(\text{by }\eqref{diez})\\
& = \frac{1}{2}\displaystyle{\sup_{\theta \in \mathbb{R}}}
{\left\|\Big(e^{i\theta}S^{\sharp_A}+e^{-i\theta}(T^{\sharp_A})^{\sharp_A}\Big)
\Big(e^{-i\theta}(S^{\sharp_A})^{\sharp_A}+e^{i\theta}T^{\sharp_A}\Big)\right\|}_A^{\frac{1}{2}}\\
& = \frac{1}{2}\displaystyle{\sup_{\theta \in \mathbb{R}}}
{\left\|S^{\sharp_A} (S^{\sharp_A})^{\sharp_A} + (T^{\sharp_A})^{\sharp_A} T^{\sharp_A}
+ 2\Re_A\big(e^{2i\theta}S^{\sharp_A}T^{\sharp_A}\big)\right\|}_A^{\frac{1}{2}}\\
& \leq \frac{1}{2}\sqrt{\left\|S^{\sharp_A} (S^{\sharp_A})^{\sharp_A} + (T^{\sharp_A})^{\sharp_A} T^{\sharp_A}\right\|_A+2\displaystyle{\sup_{\theta \in \mathbb{R}}}\left\|\Re_A\big(e^{2i\theta}S^{\sharp_A}T^{\sharp_A}\big)\right\|_A }\\
&= \frac{1}{2}\sqrt{\left\|\big(TT^{\sharp_A}+S^{\sharp_A}S\big)^{\sharp_A}\right\|_A+2\omega_A\big(S^{\sharp_A}T^{\sharp_A}\big) }\quad(\text{by Lemma } \ref{lm2})\\
&= \frac{1}{2}\sqrt{\left\|TT^{\sharp_A}+S^{\sharp_A}S\right\|_A+2\omega_A\big(TS\big) }=\Psi_A(T,S),
\end{align*}
where the last equality holds since ${\|X\|}_A = {\|X^{\sharp_A}\|}_A$ and $\omega_A(X)=\omega_A(X^{\sharp_A})$ for every $X\in\mathcal{B}_{A}(\mathcal{H})$. Hence, $\omega_{\mathbb{A}}\left[\begin{pmatrix}
0 &T\\
S &0
\end{pmatrix}\right]\leq \Psi_A(T,S)$. Moreover, by \cite{feki04} we have $\omega_{\mathbb{A}}\left[\begin{pmatrix}
0 &T\\
S &0
\end{pmatrix}\right]=\omega_{\mathbb{A}}\left[\begin{pmatrix}
0 &S\\
T&0
\end{pmatrix}\right]$. This implies that $\omega_{\mathbb{A}}\left[\begin{pmatrix}
0 &T\\
S &0
\end{pmatrix}\right]\leq \Psi_A(S,T)$. Thus, $\omega_{\mathbb{A}}\left[\begin{pmatrix}
0 &T\\
S &0
\end{pmatrix}\right]\leq \min\big\{\Psi_A(T,S),\Psi_A(S,T)\big\}$ as desired. In addition, we have
\begin{align*}
\min\big\{\Psi_A(T,S),\Psi_A(S,T)\big\}
&\leq \Psi_A(T,S)\\
 &=\frac{1}{2}\sqrt{{\big\|TT^{\sharp_A} + S^{\sharp_A}S\big\|}_A+ 2\omega_A(TS)}\\
  &\leq\frac{1}{2}\sqrt{\|TT^{\sharp_A}\|_A + \|S^{\sharp_A}S\|_A+ 2\|TS\|_A}\quad (\text{by }\,\eqref{refine1})\\
    &\leq\frac{1}{2}\sqrt{\|T\|_A^2+\|S\|_A^2+ 2\|T\|_A\|S\|_A}\;\, (\text{by }\eqref{diez}\text{ and }\eqref{sousmultiplicative})\\
        &=\frac{1}{2}\sqrt{\left(\|T\|_A+\|S\|_A\right)^2}=\frac{\|T\|_A+\|S\|_A}{2}.
\end{align*}
Hence, \eqref{r62} is proved. On the other hand, let $x\in \mathcal{H}$ be such ${\|x\|}_A = 1$ and let $\beta$ be a real number which satisfies ${\langle S^{\sharp_A} T^{\sharp_A} x\mid x\rangle}_A=e^{-2i\beta}\big|{\langle S^{\sharp_A} T^{\sharp_A} x\mid x\rangle}_A\big|$. So, we have
\begin{align}\label{maliz1}
\omega_{\mathbb{A}}\left[\begin{pmatrix}
0&T\\
S &0
\end{pmatrix}\right]
&=\omega_{\mathbb{A}}\left[\begin{pmatrix}
0&T\\
S &0
\end{pmatrix}^{\sharp_\mathbb{A}}\right]\nonumber\\
&\geq\frac{1}{2}\left\|e^{i\beta}S^{\sharp_A}+e^{-i\beta}(T^{\sharp_A})^{\sharp_A}\right\|_A\nonumber\\
& = {\left\|\Big(e^{i\beta}S^{\sharp_A}+e^{-i\beta}(T^{\sharp_A})^{\sharp_A}\Big)
\Big(e^{-i\beta}(S^{\sharp_A})^{\sharp_A}+e^{i\beta}T^{\sharp_A}\Big)\right\|}_A^{\frac{1}{2}}\quad (\text{by }\,\eqref{diez})\nonumber\\
& = \frac{1}{2}{\left\|S^{\sharp_A} (S^{\sharp_A})^{\sharp_A} + (T^{\sharp_A})^{\sharp_A} T^{\sharp_A}
+ 2\Re_A\big(e^{2i\beta}S^{\sharp_A}T^{\sharp_A}\big)\right\|}_A^{\frac{1}{2}}.
\end{align}
It can be verified that $S^{\sharp_A} (S^{\sharp_A})^{\sharp_A} + (T^{\sharp_A})^{\sharp_A} T^{\sharp_A}
+ 2\Re_A\big(e^{2i\beta}S^{\sharp_A}T^{\sharp_A}\big)$ is $A$-selfadjoint, then by \eqref{aself1} we have
\begin{align*}
&{\left\|S^{\sharp_A} (S^{\sharp_A})^{\sharp_A} + (T^{\sharp_A})^{\sharp_A} T^{\sharp_A}
+ 2\Re_A\big(e^{2i\beta}S^{\sharp_A}T^{\sharp_A}\big)\right\|}_A\\
&\geq \left|\langle\Big(S^{\sharp_A} (S^{\sharp_A})^{\sharp_A} + (T^{\sharp_A})^{\sharp_A} T^{\sharp_A}
+ 2\Re_A\big(e^{2i\beta}S^{\sharp_A}T^{\sharp_A}\big)x \Big)\mid x\rangle_A\right| \\
 &=\left|\langle\big(S^{\sharp_A} (S^{\sharp_A})^{\sharp_A} + (T^{\sharp_A})^{\sharp_A} T^{\sharp_A}\big)x\mid x\rangle_A
+ 2\langle\Re_A\big(e^{2i\beta}S^{\sharp_A}T^{\sharp_A}\big)x \mid x\rangle_A\right| \\
&=\left|\langle\big(S^{\sharp_A} (S^{\sharp_A})^{\sharp_A} + (T^{\sharp_A})^{\sharp_A} T^{\sharp_A}\big)x\mid x\rangle_A
+ 2\Re\Big(e^{2i\beta}\langle S^{\sharp_A}T^{\sharp_A}x \mid x\rangle_A\Big)\right| \\
&=\Big|\langle\big(S^{\sharp_A} (S^{\sharp_A})^{\sharp_A} + (T^{\sharp_A})^{\sharp_A} T^{\sharp_A}\big)x\mid x\rangle_A
+ 2\big|{\langle S^{\sharp_A} T^{\sharp_A} x\mid x\rangle}_A\big|\Big|\\
&=\langle\big(S^{\sharp_A} (S^{\sharp_A})^{\sharp_A} + (T^{\sharp_A})^{\sharp_A} T^{\sharp_A}\big)x\mid x\rangle_A
+ 2\big|{\langle S^{\sharp_A} T^{\sharp_A} x\mid x\rangle}_A\big|,
\end{align*}
where the last equality follows since $S^{\sharp_A} (S^{\sharp_A})^{\sharp_A} + (T^{\sharp_A})^{\sharp_A} T^{\sharp_A}\geq_A0$. So, by taking into account \eqref{maliz1} and the fact that $c_A(X)=c_A(X^{\sharp_A})$ for all $X\in \mathcal{B}_A(\mathcal{H})$, we obtain
\begin{align*}
\omega_{\mathbb{A}}\left[\begin{pmatrix}
0&T\\
S &0
\end{pmatrix}\right]
\geq\frac{1}{2}\sqrt{\langle\big(S^{\sharp_A} (S^{\sharp_A})^{\sharp_A} + (T^{\sharp_A})^{\sharp_A} T^{\sharp_A}\big)x\mid x\rangle_A
+ 2c_A(TS)}.
\end{align*}
So, by taking the supremum over all $x\in \mathcal{H}$ with $\|x\|_A=1$ in the above inequality and then using \eqref{aself1} we get
\begin{align*}
\omega_{\mathbb{A}}\left[\begin{pmatrix}
0&T\\
S &0
\end{pmatrix}\right]
\geq\frac{1}{2}\sqrt{\left\|\big(S^{\sharp_A} (S^{\sharp_A})^{\sharp_A} + (T^{\sharp_A})^{\sharp_A} T^{\sharp_A}\big)\right\|_A
+ 2c_A(TS)}=\Phi(T,S),
\end{align*}
where the last equality follows since $\|X\|_A=\|X^{\sharp_A}\|_A$ for all $X\in \mathcal{B}_A(\mathcal{H})$. Now, by using an argument similar to that used in the proof of \eqref{r62}, we get \eqref{r61} as desired. This finishes the proof of the theorem.
\end{proof}
By letting $T=S$ in the above theorem and then using \eqref{ppp}, we reach the following corollary which generalizes \cite[Theorem 2.4.]{A.K.1} and considerably improves the inequalities in Theorem \ref{main1}.
\begin{corollary}\label{corr2020}
Let $T\in\mathcal{B}_{A}(\mathcal{H})$. Then
\begin{align*}
\frac{1}{2}\sqrt{{\big\|TT^{\sharp_A} + T^{\sharp_A} T\big\|}_A + 2c_A(T^2)}\leq\omega_A(T) \leq \frac{1}{2}\sqrt{{\big\|TT^{\sharp_A} + T^{\sharp_A} T\big\|}_A + 2\omega_A(T^2)}.
\end{align*}
\end{corollary}
\begin{remark}
Obviously, the first inequality in Corollary \ref{corr2020} is sharper than the first inequality in Theorem \ref{main1}. Moreover, it can be observed that
\begin{align*}
\omega_A^2(T)
&\leq \frac{1}{4}{\big\|TT^{\sharp_A} + T^{\sharp_A} T\big\|}_A + \frac{1}{2}\omega_A(T^2)\\
&\leq \frac{1}{4}{\big\|TT^{\sharp_A} + T^{\sharp_A} T\big\|}_A + \frac{1}{2}\omega_A^2(T)\quad (\text{by }\; \eqref{apower})\\
&\leq \frac{1}{4}{\big\|TT^{\sharp_A} + T^{\sharp_A} T\big\|}_A + \frac{1}{4}{\big\|TT^{\sharp_A} + T^{\sharp_A} T\big\|}_A\quad (\text{by }\; \eqref{feki1})\\
&=\frac{1}{2}{\big\|TT^{\sharp_A} + T^{\sharp_A} T\big\|}_A.
\end{align*}
This shows that the second inequality in Corollary \ref{corr2020} refines the second inequality in Theorem \ref{main1}.
\end{remark}

As an application of Corollary \ref{corr2020} we state the following result.
\begin{theorem}
Let $\mathbb{A}=\begin{pmatrix}
A&0\\
0&A
\end{pmatrix}$. Let also $T\in\mathcal{B}_A(\mathcal{H})$ be such that $\mathcal{N}(A)^\perp$ is an invariant subspace for $T$. Then,
\begin{equation}\label{involution}
\omega_\mathbb{A}\left[\begin{pmatrix}
I&T \\
0&-I
\end{pmatrix} \right]=\frac{1}{2}\left(\left\|\begin{pmatrix}
I&T \\
0&-I
\end{pmatrix} \right\|_\mathbb{A}+\left\|\begin{pmatrix}
I&T \\
0&-I
\end{pmatrix} \right\|_\mathbb{A}^{-1}\right).
\end{equation}
\end{theorem}
\begin{proof}
Let $\mathbb{T}=\begin{pmatrix}
I&T \\
0&-I
\end{pmatrix}$. It can be seen that $\mathbb{T}^2=\begin{pmatrix}
I&0 \\
0&I
\end{pmatrix}$. Hence, by Corollary \ref{corr2020} we have
\begin{align}\label{wich}
\omega_\mathbb{A}(\mathbb{T})=\frac{1}{2}\sqrt{{\big\|\mathbb{T}\mathbb{T}^{\sharp_\mathbb{A}} + \mathbb{T}^{\sharp_\mathbb{A}} \mathbb{T}\big\|}_\mathbb{A} + 2}.
\end{align}
Now, by Lemma \ref{lm4} we have $\mathbb{T}^{\sharp_\mathbb{A}}=\begin{pmatrix}
P_{\overline{\mathcal{R}(A)}} &0 \\
T^{\sharp_A}&-P_{\overline{\mathcal{R}(A)}}
\end{pmatrix}$. Moreover, since $\mathcal{N}(A)^\perp$ is an invariant subspace for $T$, then by \cite[Lemma 1.1]{bakfeki04} we have $TP_{\overline{\mathcal{R}(A)}} =P_{\overline{\mathcal{R}(A)}}T$. Thus, a short calculation reveals that
$$\mathbb{T}\mathbb{T}^{\sharp_\mathbb{A}} + \mathbb{T}^{\sharp_\mathbb{A}} \mathbb{T}=\begin{pmatrix}
TT^{\sharp_A}+2P_{\overline{\mathcal{R}(A)}} &0 \\
0&T^{\sharp_A}T+2P_{\overline{\mathcal{R}(A)}}
\end{pmatrix}.$$
This implies, by \cite{feki04}, that
$$
\|\mathbb{T}\mathbb{T}^{\sharp_\mathbb{A}} + \mathbb{T}^{\sharp_\mathbb{A}} \mathbb{T}\|_\mathbb{A}=\max\left\{\|TT^{\sharp_A}+2P_{\overline{\mathcal{R}(A)}} \|_A ,\|T^{\sharp_A}T+2P_{\overline{\mathcal{R}(A)}} \|_A \right\}=\|T\|_A^2+2,
$$
where the last equality follows by using \eqref{aself1} since the operators $TT^{\sharp_A}+2P_{\overline{\mathcal{R}(A)}}$ and $T^{\sharp_A}T+2P_{\overline{\mathcal{R}(A)}}$ are $A$-positive. So, by taking into consideration \eqref{wich} we get
\begin{align}\label{wich2}
\omega_\mathbb{A}(\mathbb{T})=\frac{1}{2}\sqrt{\|T\|_A^2+4}.
\end{align}
Now, we will prove that
\begin{align}\label{wich3}
\|\mathbb{T}\|_\mathbb{A}=\frac{1}{\sqrt{2}}\sqrt{2+\|T\|_A^2+\sqrt{\|T\|_A^4+4\|T\|_A^2} }=\frac{1}{2}\sqrt{\|T\|_A^2+4}+\frac{1}{2}\|T\|_A.
\end{align}
By \eqref{newsemi} there exists two sequences of $A$-unit vectors $\{x_n\}$ and $\{y_n\}$ in $\mathcal{H}$ such that
$$\lim_{n\to +\infty}|\langle Tx_n\mid y_n\rangle_A|=\|T\|_A.$$
This implies, by applying the Cauchy-Schwarz inequality, that $\displaystyle\lim_{n\to +\infty}\|Tx_n\|_A=\|T\|_A.$ Let $(a,b)\in \mathbb{R}^2$ be such that $a^2+b^2=1$ and
\begin{equation}\label{matrixnorm}
\left\|\begin{pmatrix}
1&\|T\|_A \\
0&1
\end{pmatrix}\right\|=\left\|\begin{pmatrix}
1&\|T\|_A \\
0&1
\end{pmatrix}\begin{pmatrix}a\\b\end{pmatrix}\right\|=\sqrt{\left(a+b\|T\|_A\right)^2+b^2}.
\end{equation}
For $n\in \mathbb{N}$, let $\langle Tx_n\mid y_n\rangle_A=|\langle Tx_n\mid y_n\rangle_A|\,e^{i\alpha_n}$ for some $\alpha_n\in \mathbb{R}$. Let $\{X_{n}\}=\{(ae^{i\alpha_n}y_n,bx_n)^T\}$ be a sequence in $\mathcal{H}\oplus \mathcal{H}$. It is not difficult to see that $\|X_{n}\|_\mathbb{A}=1$. Moreover, a short calculation reveals that
\begin{align*}
\left\|\begin{pmatrix}
I&T \\
0&-I
\end{pmatrix}\right\|_\mathbb{A}^2
&\geq \left\|\begin{pmatrix}
I&T \\
0&-I
\end{pmatrix}\begin{pmatrix}
ae^{i\alpha_n}y_n\\
bx_n
\end{pmatrix}\right\|_\mathbb{A}^2\\
 &=\|ae^{i\alpha_n}y_n+bTx_n\|_A^2+\|bx_n\|_A^2\\
 &=a^2+b^2\|Tx_n\|_A^2+2ab|\langle Tx_n\mid y_n\rangle_A|+|b|^2\\
  &\xrightarrow{n \to +\infty} a^2+b^2\|T\|_A^2+2ab\|T\|_A+|b|^2\\
  &=\left\|\begin{pmatrix}
1&\|T\|_A \\
0&1
\end{pmatrix}\right\|^2\quad (\text{by } \eqref{matrixnorm}).
\end{align*}
Hence,
$$
\left\|\begin{pmatrix}
I&T \\
0&-I
\end{pmatrix}\right\|_\mathbb{A}\geq \left\|\begin{pmatrix}
1&\|T\|_A \\
0&1
\end{pmatrix}\right\|.
$$
Moreover, by \cite{feki02} we have
$$
\left\|\begin{pmatrix}
I&T \\
0&-I
\end{pmatrix}\right\|_\mathbb{A}\leq \left\|\begin{pmatrix}
1&\|T\|_A \\
0&1
\end{pmatrix}\right\|.
$$
Hence, $\|\mathbb{T}\|_\mathbb{A}=\left\|\begin{pmatrix}
1&\|T\|_A \\
0&1
\end{pmatrix}\right\|$. On the other, by using \eqref{absolute}, we see that
\begin{align*}
\left\|\begin{pmatrix}
1&\|T\|_A \\
0&1
\end{pmatrix}\right\|^2
& =r\left[\begin{pmatrix}
1& 0\\
\|T\|_A&1
\end{pmatrix}\begin{pmatrix}
1&\|T\|_A \\
0&1
\end{pmatrix}\right] \\
 &=r\left[\begin{pmatrix}
1& \|T\|_A\\
\|T\|_A&\|T\|_A^2+1
\end{pmatrix}\right]\\
 &=\frac{1}{2}\left(2+\|T\|_A^2+\sqrt{\|T\|_A^4+4\|T\|_A^2}\right).
\end{align*}
Thus, we prove the first equality in \eqref{wich3}. The second equality in \eqref{wich3} follows immediately. Therefore, \eqref{wich3} is proved. Now, by combining \eqref{wich2} together with \eqref{wich3}, we get \eqref{involution} as desired.
\end{proof}

\begin{corollary}
Let $\mathbb{A}=\begin{pmatrix}
A&0\\
0&A
\end{pmatrix}$. Let also $\mathbb{T}=\begin{pmatrix}
I&T \\
0&-I
\end{pmatrix}$ be such that $T\in\mathcal{B}_A(\mathcal{H})$ and $\mathcal{N}(A)^\perp$ is an invariant subspace for $T$. Then, the following assertions hold
\begin{itemize}
  \item [(1)] $\|\Re_A(\mathbb{T})\|_\mathbb{A}=\omega_\mathbb{A}(\mathbb{T})$.
  \item [(2)] $\|\Im_A(\mathbb{T})\|_\mathbb{A}=\frac{1}{2}\left(\|\mathbb{T}\|_\mathbb{A}-\|\mathbb{T}\|_\mathbb{A}^{-1}\right)$.
\end{itemize}
\end{corollary}
\begin{proof}
\noindent (1)\;It can be seen that
\begin{align*}
\|\Re_A(\mathbb{T})\|_\mathbb{A}^2
& =\left\|[\Re_A(\mathbb{T})]^{\sharp_{\mathbb{A}}}\right\|_\mathbb{A}^2\\
& =\left\|\begin{pmatrix}
P_{\overline{\mathcal{R}(A)}} &\frac{(T^{\sharp_A})^{\sharp_A}}{2}\\
\frac{T^{\sharp_A}}{2}&-P_{\overline{\mathcal{R}(A)}}
\end{pmatrix}\right\|_\mathbb{A}^2 \\
 &=\left\|\begin{pmatrix}
P_{\overline{\mathcal{R}(A)}} &\frac{(T^{\sharp_A})^{\sharp_A}}{2}\\
\frac{T^{\sharp_A}}{2}&-P_{\overline{\mathcal{R}(A)}}
\end{pmatrix}
\begin{pmatrix}
P_{\overline{\mathcal{R}(A)}} &\frac{(T^{\sharp_A})^{\sharp_A}}{2}\\
\frac{T^{\sharp_A}}{2}&-P_{\overline{\mathcal{R}(A)}}
\end{pmatrix}
\right\|_\mathbb{A}\quad (\text{by }\,\eqref{diez})\\
& =\left\|\begin{pmatrix}
P_{\overline{\mathcal{R}(A)}} +\frac{(T^{\sharp_A})^{\sharp_A}T^{\sharp_A}}{4}&0\\
0&P_{\overline{\mathcal{R}(A)}} +\frac{T^{\sharp_A}(T^{\sharp_A})^{\sharp_A}}{4}
\end{pmatrix}
\right\|_\mathbb{A},
\end{align*}
where the last equality follows by using the fact that $\mathcal{N}(A)^\perp$ is an invariant subspace for $T$. So, we obtain
\begin{align*}
\|\Re_A(\mathbb{T})\|_\mathbb{A}^2
& =\max\left\{\left\|P_A+\frac{(T^{\sharp_A})^{\sharp_A}T^{\sharp_A}}{4}\right\|_A ,\left\|P_A+\frac{T^{\sharp_A}(T^{\sharp_A})^{\sharp_A}}{4}\right\|_A\right\} \\
 &=\frac{\|T^{\sharp_A}\|_A^2}{4}+1=\frac{\|T\|_A^2+4}{4}.
\end{align*}
Hence, we prove the desired property.
\par \vskip 0.1 cm \noindent (2)\;By proceeding as in (1), one can prove that $\|\Im_A(\mathbb{T})\|_\mathbb{A}=\frac{1}{2}\|T\|_A$. So, the required result holds by using \eqref{wich3} and \eqref{involution}.
\end{proof}

Now, we turn your attention to the study of some $A$-numerical radius for products and commutators of semi-Hilbert space operators. Our first result in this context generalizes \cite[Corollary 4.3]{zamani1} since $\mathcal{B}_{A^{1/2}}(\mathcal{H})\subseteq \mathcal{B}_{A}(\mathcal{H})$ and reads as follows.
\begin{theorem}
Let $T,S\in \mathcal{B}_{A^{1/2}}(\mathcal{H})$. Then,
\begin{equation}\label{a7ad1}
\omega_A(TS)\leq 4\,\omega_A(T)\,\omega_A(S).
\end{equation}
If $TS=ST$, then
\begin{equation}\label{a7ad2}
\omega_A(TS)\leq 2\omega_A(T)\,\omega_A(S).
\end{equation}
\end{theorem}
\begin{proof}
It follows from \eqref{refine1} and \eqref{sousmultiplicative} that
$$\omega_A(TS)\leq\|TS\|_A\leq \|T\|_A\|S\|_A\leq 4\,\omega_A(T)\omega_A(S).$$
This proves \eqref{a7ad1}. Now, in order to establish \eqref{a7ad2}, we first prove that
\begin{align}\label{jdid}
\omega_A(TS +ST) \leq4\,\omega_A(T)\omega_A(S).
\end{align}
Assume that $T,S\in \mathcal{B}_{A^{1/2}}(\mathcal{H})$ and satisfy $\omega_A(T)=\omega_A(S)=1$. It is not difficult to observe that
$$TS +ST=\frac{1}{2}\left[(T+S)^2-(T-S)^2\right].$$
So, by using the fact that $\omega_A(\cdot)$ is a seminorm and \eqref{apower} we see that
\begin{align*}
\omega_A(TS +ST)
&\leq \frac{1}{2}\left[\omega_A^2(T+S)+\omega_A^2(T-S)\right]\\
 & \leq \left[\omega_A(T)+\omega_A(S)\right]^2=4.
\end{align*}
Hence,
\begin{align}\label{equal1}
\omega_A(TS +ST)
 & \leq 4,
\end{align}
for all $T, S\in\mathcal{B}_{A^{1/2}}(\mathcal{H})$ with $\omega_A(T)=\omega_A(S)=1$. If $\omega_A(T)=\omega_A(S)=0$, then $AT=AS=0$ and so \eqref{jdid} holds trivially. Now, assume that $\omega_A(T)\neq0$ and $\omega_A(S)\neq0$. By replacing $T$ and $S$ by $\frac{T}{\omega_A(T)}$ and $\frac{S}{\omega_A(S)}$ respectively in \eqref{equal1}, we get \eqref{jdid} as required. Now, if $TS=ST$, then \eqref{a7ad2} follows immediately from \eqref{jdid}.
\end{proof}

Our next result reads as follows.
\begin{theorem}\label{T.2.15}
Let $T_1, T_2,S_1,S_2\in\mathcal{B}_{A}(\mathcal{H})$. Then
\begin{align*}
\omega_A(T_1S_1 \pm S_2T_2) &\leq \sqrt{{\big\|T_1T_1^{\sharp_A} + T_2^{\sharp_A} T_2\big\|}_A}\sqrt{{\big\|S_1^{\sharp_A} S_1+S_2S_2^{\sharp_A}\big\|}_A}.
\end{align*}
\end{theorem}
\begin{proof}
Let $x\in \mathcal{H}$ be such that ${\|x\|}_A = 1$. An application of the Cauchy-Schwarz inequality gives
\begin{align*}
\Big|{\big\langle (T_1S_1 \pm S_2T_2) x\mid x \big\rangle}_A\Big|^2
&\leq \Big(\big|{\langle T_1S_1 x\mid x\rangle}_A\big| + \big|{\langle S_2T_2 x\mid x\rangle}_A\big|\Big)^2\\
& = \Big(\big|{\langle S_1 x\mid T_1^{\sharp_A} x\rangle}_A\big| + \big|{\langle T_2 x\mid S_2^{\sharp_A} x\rangle}_A\big|\Big)^2\\
& \leq \Big({\|S_1x\|}_A {\|T_1^{\sharp_A} x\|}_A + {\|T_2x\|}_A{\|S_2^{\sharp_A} x\|}_A\Big)^2\\
& \leq \Big({\|T_2x\|}^2_A  + {\|T_1^{\sharp_A} x\|}^2_A\Big)\Big({\|S_1x\|}^2_A + {\|S_2^{\sharp_A} x\|}^2_A\Big)\\
& = {\langle  \big(T_2^{\sharp_A}T_2 + T_1T_1^{\sharp_A}\big)x\mid x\rangle}_A{\langle \big(S_1^{\sharp_A}S_1 + S_2S_2^{\sharp_A}\big) x\mid x\rangle}_A\\
& \leq {\big\|T_1T_1^{\sharp_A} + T_2^{\sharp_A} T_2\big\|}_A{\big\|S_1^{\sharp_A} S_1+S_2S_2^{\sharp_A}\big\|}_A.
\end{align*}
Thus
\begin{align*}
\Big|{\big\langle (T_1S_1 \pm S_2T_2) x\mid x \big\rangle}_A\Big|
\leq {\big\|T_1T_1^{\sharp_A} + T_2^{\sharp_A} T_2\big\|}_A^{1/2}{\big\|S_1^{\sharp_A} S_1+S_2S_2^{\sharp_A}\big\|}_A^{1/2},
\end{align*}
for all $x\in \mathcal{H}$ with ${\|x\|}_A = 1$. By taking the supremum over all $x\in \mathcal{H}$ with ${\|x\|}_A = 1$ in the above inequality we obtain
\begin{align*}
\omega_A(T_1S_1 \pm S_2T_2) &\leq \sqrt{{\big\|T_1T_1^{\sharp_A} + T_2^{\sharp_A} T_2\big\|}_A}\sqrt{{\big\|S_1^{\sharp_A} S_1+S_2S_2^{\sharp_A}\big\|}_A},
\end{align*}
as desired.
\end{proof}

The following corollary is an immediate consequence of Theorem \ref{T.2.15} and generalizes the well-known result proved by Fong and Holbrook in \cite{fong}.
\begin{corollary}
Let $T, S\in\mathcal{B}_{A}(\mathcal{H})$. Then
\begin{align}\label{commu223}
\omega_A(TS \pm ST) &\leq 2\sqrt{2}\min\Big\{{\|T\|}_A \omega_A(S), {\|S\|}_A \omega_A(T)\Big\},
\end{align}
\end{corollary}
\begin{proof}
By letting $T_1=T_2=T$ and $S_1=S_2=S$ in Theorem \ref{T.2.15} and then using the second inequality in \eqref{feki1} we get
\begin{align*}
\omega_A(TS \pm ST)
&\leq \sqrt{{\big\|TT^{\sharp_A} + T^{\sharp_A} T\big\|}_A}\sqrt{{\big\|SS^{\sharp_A} + S^{\sharp_A} S\big\|}_A}\\
&\leq 2\sqrt{\big\|TT^{\sharp_A}\big\|_A^2+\big\|T^{\sharp_A} T\big\|_A^2}\,\omega_A(S)\\
&= 2\sqrt{\big\|T\big\|_A^2+\big\|T\big\|_A^2}\,\omega_A(S)\quad (\text{by }\,\eqref{diez}).
\end{align*}
Hence,
\begin{equation}\label{pp}
\omega_A(TS \pm ST)\leq 2\sqrt{2}\big\|T\big\|_A\,\omega_A(S).
\end{equation}
By replacing $T$ and $S$ by $S$ and $T$ respectively in \eqref{pp}, we get the desired result.
\end{proof}
Clearly \eqref{commu223} provides an upper bound for the $A$-numerical radius of the commutator $TS-ST$.

Next, we state the following result which is a natural generalization of another well-known theorem proved by Fong and Holbrook in \cite{fong}.
\begin{theorem}
Let $T, S\in\mathcal{B}_{A}(\mathcal{H})$. Then,
\begin{align*}
\omega_A(TS\pm ST^{\sharp_A}) \leq2\|T\|_A\,\omega_A(S).
\end{align*}
\end{theorem}
\begin{proof}
We first prove that
\begin{align}\label{initial1}
\omega_A(TS+ ST^{\sharp_A}) \leq2\|T\|_A\,\omega_A(S).
\end{align}
In view of Lemma \ref{lm2}, we have
\begin{align}\label{777aaaaa}
\omega_A(TS+ ST^{\sharp_A})
& = \omega_A([TS+ ST^{\sharp_A}]^{\sharp_A})\nonumber\\
 &=\displaystyle{\sup_{\theta \in \mathbb{R}}}{\left\|\Re_A\left(e^{i\theta}[TS+ ST^{\sharp_A}]^{\sharp_A}\right)\right\|}_A.
\end{align}
Moreover, a short calculation reveals that
$$\Re_A\left(e^{i\theta}[TS+ ST^{\sharp_A}]^{\sharp_A}\right)= (T^{\sharp_A})^{\sharp_A} \left[\Re_A\left(e^{i\theta}S^{\sharp_A}\right)\right]+ \left[\Re_A\left(e^{i\theta}S^{\sharp_A}\right)\right]T^{\sharp_A},$$
for every $\theta \in \mathbb{R}$. So, by using \eqref{777aaaaa} together with Lemma \eqref{lm2} we get
\begin{align*}
\omega_A(TS+ ST^{\sharp_A})
 &=\displaystyle{\sup_{\theta \in \mathbb{R}}}\left\|(T^{\sharp_A})^{\sharp_A} \left[\Re_A\left(e^{i\theta}S^{\sharp_A}\right)\right]+ \left[\Re_A\left(e^{i\theta}S^{\sharp_A}\right)\right]T^{\sharp_A}\right\|_A\\
  &\leq\|(T^{\sharp_A})^{\sharp_A} \|_A\left(\displaystyle{\sup_{\theta \in \mathbb{R}}}\left\|\Re_A\left(e^{i\theta}S^{\sharp_A}\right)\right\|_A\right)+\|T^{\sharp_A}\|_A\left(\displaystyle{\sup_{\theta \in \mathbb{R}}}\left\|\Re_A\left(e^{i\theta}S^{\sharp_A}\right)\right\|_A\right)\\
    &=2\|T\|_A\,\omega_A(S^{\sharp_A})\;\; (\text{by } \eqref{diez})\\
        &=2\|T\|_A\,\omega_A(S).
\end{align*}
This proves \eqref{initial1}. By replacing $T$ by $iT$ in \eqref{initial1} we get
\begin{align}
\omega_A(TS-ST^{\sharp_A}) \leq2\|T\|_A\,\omega_A(S).
\end{align}
Therefore, we prove the desired result.
\end{proof}

Our next aim is to give an improvement of \eqref{kaisnew01}. In order to achieve this goal and other goals in the rest of this paper, we need the following lemmas.
\begin{lemma}$($\cite{acg3}$)$\label{lr1}
	Let $T\in \mathcal{B}(\mathcal{H})$. Then $T\in \mathcal{B}_{A^{1/2}}(\mathcal{H})$ if and only if there exists a unique $\widetilde{T}\in \mathcal{B}(\mathbf{R}(A^{1/2}))$ such that $Z_AT =\widetilde{T}Z_A$. Here, $Z_{A}: \mathcal{H} \rightarrow \mathbf{R}(A^{1/2})$ is
	defined by $Z_{A}x = Ax$.
\end{lemma}

\begin{lemma}$($\cite{feki01}$)$\label{lr2}
	Let $T\in \mathcal{B}_{A^{1/2}}(\mathcal{H})$. Then
	\begin{itemize}
		\item [(a)] $\|T\|_A=\|\widetilde{T}\|_{\mathcal{B}(\mathbf{R}(A^{1/2}))}$.
		\item [(b)] $\omega_A(T)=\omega(\widetilde{T})$.
	\end{itemize}
\end{lemma}

\begin{lemma}\label{lr3}$($\cite[Proposition 2.9]{majsecesuci}$)$
	Let $T\in \mathcal{B}_A(\mathcal{H})$. Then
	$$\widetilde{T^{\sharp_A}}=\big(\widetilde{T}\big)^*\;\text{ and }\; \widetilde{({T^{\sharp_A}})^{\sharp_A}}=\widetilde{T}.$$
\end{lemma}

Now, we are ready to prove the following proposition.
\begin{proposition}\label{immmmm}
Let $T\in \mathcal{B}_A(\mathcal{H})$. Then,
\begin{equation}\label{omprovenew}
2\|T^2\|_A\leq \|T^{\sharp_A} T+TT^{\sharp_A}\|_A\leq \|T^2\|_A+\|T\|_A^2.
\end{equation}
\end{proposition}
\begin{proof}
Firstly, we shall prove that for every $X,Y\in \mathcal{B}_{A^{1/2}}(\mathcal{H})$ we have
\begin{equation}\label{prosum}
\widetilde{XY}=\widetilde{X}\widetilde{Y}\;\text{ and }\; \widetilde{X+Y}=\widetilde{X}+\widetilde{Y}.
\end{equation}
Since, $X,Y\in \mathcal{B}_{A^{1/2}}(\mathcal{H})$, then by Lemma \ref{lr1} there exists a unique $\widetilde{X},\widetilde{Y}\in \mathcal{B}(\mathbf{R}(A^{1/2}))$ such that $Z_AX =\widetilde{X}Z_A$ and $Z_AY=\widetilde{Y}Z_A$. So,
$$Z_A(XY) =\widetilde{X}Z_AY=\widetilde{X}\widetilde{Y}Z_A.$$
On the other hand, since $XY\in \mathcal{B}_{A^{1/2}}(\mathcal{H})$, then again in view of Lemma \ref{lr1} we have $Z_A(XY) =(\widetilde{XY})Z_A$. Thus, since $\widetilde{XY}$ is unique, then we conclude that $\widetilde{XY}=\widetilde{X}\widetilde{Y}$. Similarly, we can prove that $\widetilde{X+Y}=\widetilde{X}+\widetilde{Y}$.

Now, by using Lemma \ref{lr1} (a) together with \eqref{prosum} and Lemma \ref{lr3} we get
\begin{align}\label{firr}
\|T^{\sharp_A} T+TT^{\sharp_A}\|_A
& =\|\widetilde{T^{\sharp_A} T+TT^{\sharp_A}}\|_{\mathcal{B}(\mathbf{R}(A^{1/2}))} \nonumber\\
 &=\|(\widetilde{T})^{*}\widetilde{T}+\widetilde{T}(\widetilde{T})^{*}\|_{\mathcal{B}(\mathbf{R}(A^{1/2}))}.
\end{align}
Moreover, by using basic properties of the spectral radius of Hilbert space operators, we see that
\begin{align*}
\|(\widetilde{T})^{*}\widetilde{T}+\widetilde{T}(\widetilde{T})^{*}\|_{\mathcal{B}(\mathbf{R}(A^{1/2}))}
& =r\left((\widetilde{T})^{*}\widetilde{T}+\widetilde{T}(\widetilde{T})^{*}\right)\\
 &=r\left[\begin{pmatrix}
(\widetilde{T})^{*}\widetilde{T}+\widetilde{T}(\widetilde{T})^{*}&0 \\
0&0
\end{pmatrix}\right]\\
&=r\left[\begin{pmatrix}
|\widetilde{T}|&|(\widetilde{T})^{*}| \\
0&0
\end{pmatrix}
\begin{pmatrix}
|\widetilde{T}|&0 \\
|(\widetilde{T})^{*}|&0
\end{pmatrix}
\right]\\
&=r\left[
\begin{pmatrix}
|\widetilde{T}|&0 \\
|(\widetilde{T})^{*}|&0
\end{pmatrix}
\begin{pmatrix}
|\widetilde{T}|&|(\widetilde{T})^{*}|\\
0&0
\end{pmatrix}
\right]
\end{align*}
Hence, we get
\begin{align*}
\|(\widetilde{T})^{*}\widetilde{T}+\widetilde{T}(\widetilde{T})^{*}\|_{\mathcal{B}(\mathbf{R}(A^{1/2}))}=r\left[
\begin{pmatrix}
(\widetilde{T})^{*}\widetilde{T}& |\widetilde{T}|\,|(\widetilde{T})^{*}|\\
|(\widetilde{T})^{*}|\,|\widetilde{T}|&\widetilde{T}(\widetilde{T})^{*}
\end{pmatrix}
\right]
\end{align*}
Thus, by using \cite[Theorem 1.1.]{hou} we obtain
\begin{align*}
\|(\widetilde{T})^{*}\widetilde{T}+\widetilde{T}(\widetilde{T})^{*}\|_{\mathcal{B}(\mathbf{R}(A^{1/2}))}
&\leq r\left[
\begin{pmatrix}
\|(\widetilde{T})^{*}\widetilde{T}\|_{\mathcal{B}(\mathbf{R}(A^{1/2}))}& \|\,|\widetilde{T}|\,|(\widetilde{T})^{*}|\,\|_{\mathcal{B}(\mathbf{R}(A^{1/2}))}\\
\|\,|(\widetilde{T})^{*}|\,|\widetilde{T}|\,\|_{\mathcal{B}(\mathbf{R}(A^{1/2}))}&\|\widetilde{T}(\widetilde{T})^{*}\|_{\mathcal{B}(\mathbf{R}(A^{1/2}))}
\end{pmatrix}
\right]\\
&= r\left[
\begin{pmatrix}
\|\widetilde{T}\|_{\mathcal{B}(\mathbf{R}(A^{1/2}))}^2& \|(\widetilde{T})^2\|_{\mathcal{B}(\mathbf{R}(A^{1/2}))}\\
\|(\widetilde{T})^2\|_{\mathcal{B}(\mathbf{R}(A^{1/2}))}&\|\widetilde{T}\|_{\mathcal{B}(\mathbf{R}(A^{1/2}))}^2
\end{pmatrix}
\right]\\
&=\|\widetilde{T}\|_{\mathcal{B}(\mathbf{R}(A^{1/2}))}^2+\|(\widetilde{T})^2\|_{\mathcal{B}(\mathbf{R}(A^{1/2}))}.
\end{align*}
So, by taking into account \eqref{firr} we get
\begin{align*}
\|T^{\sharp_A} T+TT^{\sharp_A}\|_A
&\leq \|\widetilde{T}\|_{\mathcal{B}(\mathbf{R}(A^{1/2}))}^2+\|(\widetilde{T})^2\|_{\mathcal{B}(\mathbf{R}(A^{1/2}))}\\
 &=\|\widetilde{T}\|_{\mathcal{B}(\mathbf{R}(A^{1/2}))}^2+\|\widetilde{T^2}\|_{\mathcal{B}(\mathbf{R}(A^{1/2}))}\quad (\text{by }\,\eqref{prosum})\\
 &=\|T\|_A^2+\|T^2\|_A\quad(\text{by Lemma } \ref{lr2} (b)).
\end{align*}
Now, by \cite[Lemma 7]{FK2002} we have
$$
2\|(\widetilde{T})^2\|_{\mathcal{B}(\mathbf{R}(A^{1/2}))}\leq \|(\widetilde{T})^{*}\widetilde{T}+\widetilde{T}(\widetilde{T})^{*}\|_{\mathcal{B}(\mathbf{R}(A^{1/2}))}.
$$
This implies, by using \eqref{prosum} together with Lemmas \ref{lr3} and \ref{lr2} (b), that
$$2\|T^2\|_A\leq\|T^{\sharp_A} T+TT^{\sharp_A}\|_A.$$
Hence, the proof is finished.
\end{proof}
Now, we state the following two corollaries. The first one give a considerable improvement of \eqref{kaisnew01}.
\begin{corollary}
Let $T\in\mathcal{B}_{A}(\mathcal{H})$. Then
\begin{align*}
\omega_A(T) \leq \frac{1}{2}\sqrt{\|T^2\|_A+\|T\|_A^2 + 2\omega_A(T^2)}\leq\frac{1}{2}\left(\|T\|_A+\|T^2\|_A^{1/2}\right).
\end{align*}
\end{corollary}
\begin{proof}
Observe first that, in view of \eqref{sousmultiplicative}, we have
\begin{equation}\label{aself12020}
\|T^2\|_A=\|T^2\|_A^{1/2}\|T^2\|_A^{1/2}\leq \|T\|_A\|T^2\|_A^{1/2}.
\end{equation}
Moreover, by Corollary \ref{corr2020}, one has
\begin{align*}
\omega_A(T)
&\leq \frac{1}{2}\sqrt{{\big\|TT^{\sharp_A} + T^{\sharp_A} T\big\|}_A + 2\omega_A(T^2)}\\
&\leq \frac{1}{2}\sqrt{\|T^2\|_A+\|T\|_A^2 + 2\omega_A(T^2)}\quad (\text{by }\,\eqref{omprovenew})\\
&\leq \frac{1}{2}\sqrt{\|T^2\|_A+\|T\|_A^2 + 2\|T^2\|_A}\quad (\text{by }\,\eqref{refine1})\\
&\leq \frac{1}{2}\sqrt{\|T\|_A^2 + 2\|T\|_A\|T^2\|_A^{1/2}+\|T^2\|_A}\quad (\text{by }\,\eqref{sousmultiplicative}\text{ and }\eqref{aself12020})\\
&=\frac{1}{2}\sqrt{\left(\|T\|_A+\|T^2\|_A^{1/2}\right)^2}\\
&=\frac{1}{2}\left(\|T\|_A+\|T^2\|_A^{1/2}\right).
\end{align*}
This completes the proof.
\end{proof}

\begin{corollary}
Let $T\in\mathcal{B}_{A}(\mathcal{H})$. Then the following assertions hold:
\begin{itemize}
  \item [(1)] $\|T\|_A=\frac{\sqrt{2}}{2}\sqrt{{\big\|T T^{\sharp_A} + T^{\sharp_A} T\big\|}_A}$ if and only if $\|T^2\|_A=\|T\|_A^2$.
  \item [(2)] If $AT^2=0$, then $\|T\|_A=\sqrt{{\big\|T T^{\sharp_A} + T^{\sharp_A} T\big\|}_A}$.
\end{itemize}
\end{corollary}
\begin{proof}
\noindent (1)\;By Proposition \ref{immmmm} we have
\begin{equation*}
2\|T\|_A^2\leq \|T^{\sharp_A} T+TT^{\sharp_A}\|_A\leq \|T^2\|_A+\|T\|_A^2\leq 2\|T\|_A^2.
\end{equation*}
So, the desired property follows.
\par \vskip 0.1 cm \noindent (2)\;Clearly, we have $(T^{\sharp_A} T+TT^{\sharp_A})-TT^{\sharp_A}=T^{\sharp_A} T\geq_A0$. Hence, by using \eqref{aself1} together with \eqref{diez} and Proposition \ref{immmmm} we obtain
\begin{equation*}
\|T\|_A^2\leq \|T^{\sharp_A} T+TT^{\sharp_A}\|_A\leq \|T^2\|_A+\|T\|_A^2.
\end{equation*}
So, if $AT^2=$ then $\|T^2\|_A=0$. Therefore, the required assertion follows immediately.
\end{proof}

By using the connection between $A$-bounded operators and operators in $\mathcal{B}(\mathbf{R}(A^{1/2}))$ we will generalize some well-known numerical radius inequalities. We begin with the following result which generalizes \cite[Theorem 3.3.]{zamani2} since $\mathcal{B}_{A}(\mathcal{H})\subseteq\mathcal{B}_{A^{1/2}}(\mathcal{H})$.

\begin{theorem}\label{newdooo}
Let $T\in \mathcal{B}_{A^{1/2}}(\mathcal{H})$. Then,
\begin{align}\label{eqnew1.5}
\omega_A^2(T) \le \tfrac{1}{2}\left[ {\|T\|_A^2 + \omega_A\big( {T^2} \big)} \right].
\end{align}
\end{theorem}
\begin{proof}
By applying Lemma \ref{lr2} and \cite[Theorem1.]{D4} we see that
\begin{align*}
\omega_A^2(T)
& = \omega^2(\widetilde{T}) \\
 &\le \frac{1}{2}\left[ {\|\widetilde{T}\|_{\mathcal{B}(\mathbf{R}(A^{1/2}))}^2 + \omega\big( {(\widetilde{T})^2} \big)} \right]\\
  &=\frac{1}{2}\left[ {\|\widetilde{T}\|_{\mathcal{B}(\mathbf{R}(A^{1/2}))}^2 + \omega\big( {\widetilde{T^2}} \big)} \right].
\end{align*}
So, we get \eqref{eqnew1.5} by applying Lemma \ref{lr2}.
\end{proof}
\begin{remark}
For a given $T\in \mathcal{B}_{A^{1/2}}(\mathcal{H})$, it follows from Theorem \ref{newdooo}, \eqref{refine1} and \eqref{sousmultiplicative} that
$$\omega_A(T) \le \left\{\tfrac{1}{2}\left[ {\|T\|_A^2 + \omega_A\big( {T^2} \big)} \right]\right\}^{1/2}\leq \left\{\tfrac{1}{2}\left[ {\|T\|_A^2 + \|T\|_A^2} \right]\right\}^{1/2}\leq \|T\|_A.$$
Hence, \eqref{eqnew1.5} refines the second inequality in \eqref{refine1}.
\end{remark}
Next, we state the following theorem.
\begin{theorem}\label{prr2}
Let $U\in\mathcal{B}_A(\mathcal{H})$ be an $A$-unitary operator and $T\in \mathcal{B}_{A^{1/2}}(\mathcal{H})$ be such that $UT=TU$. Then,
 \begin{equation}\label{hooknew02}
\omega_A(UT)\leq \omega_A(T).
 \end{equation}
If $U$ is an $A$-isometry operator which commutes with an operator $T\in \mathcal{B}_{A^{1/2}}(\mathcal{H})$, then \eqref{hooknew02} also holds true.
\end{theorem}
\begin{proof}
Notice first that, it was proved in \cite{acg1}, that an operator $U\in \mathcal{B}_A(\mathcal{H})$ is $A$-unitary if and only if
$$U^{\sharp_A} U=(U^{\sharp_A})^{\sharp_A} U^{\sharp_A}=P_{\overline{\mathcal{R}(A)}}.$$
This implies that
$$\widetilde{U^{\sharp_A} U}=\widetilde{(U^{\sharp_A})^{\sharp_A} U^{\sharp_A}}=\widetilde{P_{\overline{\mathcal{R}(A)}}}.$$
On the other hand, it can be seen that $\widetilde{P_{\overline{\mathcal{R}(A)}}}=I_{\mathbf{R}(A^{1/2}}$. So, by using Lemma \ref{lr3} we get
$$\widetilde{U}^* \widetilde{U}=\widetilde{U} \widetilde{U}^*=I_{\mathbf{R}(A^{1/2}}.$$
So, $\widetilde{U}$ is an unitary operator on the Hilbert space $\mathbf{R}(A^{1/2})$. Moreover, since $UT=TU$, then $\widetilde{U}\widetilde{T}=\widetilde{T}\widetilde{U}$. Thus, by applying \cite[Theorem 1.9.]{D.1}, we obtain
 \begin{equation*}
\omega(\widetilde{U}\widetilde{T})\leq \omega(\widetilde{T}).
 \end{equation*}
This proves \eqref{hooknew02} by observing that $\widetilde{U}\widetilde{T}=\widetilde{UT}$ and using Lemma \ref{lr2} (b). Now, let $U\in \mathcal{B}_A(\mathcal{H})$ is an $A$-isometry operator. Then, by \cite[Corollary 3.7.]{acg1}, we have $U^{\sharp_A} U=P_{\overline{\mathcal{R}(A)}}.$ By using similar arguments as above, one can see that $\widetilde{U}\in \mathcal{B}(\mathbf{R}(A^{1/2}))$ is an isometry operator. So, the proof of the theorem is complete by proceeding as above and using \cite[Theorem 1.9.]{D.1}.
\end{proof}

Our next result reads as follows.
\begin{theorem}\label{doublecomm}
Let $T\in \mathcal{B}_A(\mathcal{H})$ and $S\in  \mathcal{B}_{A^{1/2}}(\mathcal{H})$ be such that $TS=ST$ and $T^{\sharp_A} S=ST^{\sharp_A} $. Then,
 \begin{equation}\label{hooknew02222}
\omega_A(TS)\leq \omega_A(S)\|T\|_A,
 \end{equation}
 and
  \begin{equation*}\label{hooknew02222+}
\omega_A(ST)\leq \omega_A(T)\|S\|_A.
 \end{equation*}
\end{theorem}
\begin{proof}
Since $TS=ST$ and $T^{\sharp_A} S=ST^{\sharp_A} $, then by the same arguments used in the proof of Theorem \ref{prr2}, we infer that
$$\widetilde{T}\widetilde{S}=\widetilde{S}\widetilde{T}\;\text{ and }\; \widetilde{T}^* \widetilde{S}=\widetilde{S}\widetilde{T}^*.$$
So, an application of \cite[Theorem 1.10]{D.1} shows that
$$\omega(\widetilde{T}\widetilde{S})\leq \omega(\widetilde{S})\|\widetilde{T}\|_{\mathcal{B}(\mathbf{R}(A^{1/2}))}\;\text{ and } \; \omega(\widetilde{S}\widetilde{T})\leq \omega(\widetilde{T})\|\widetilde{S}\|_{\mathcal{B}(\mathbf{R}(A^{1/2}))}.$$
Therefore, the desired results by applying \eqref{prosum} together with Lemma \ref{lr2}.
\end{proof}

\begin{corollary}
Let $T,S\in \mathcal{B}_A(\mathcal{H})$ be such that $T$ is an $A$-isometry operator. Assume that $TS=ST$ and $T^{\sharp_A} S=ST^{\sharp_A} $. Then,
 \begin{equation*}
\omega_A(TS)\leq \omega_A(S).
 \end{equation*}
\end{corollary}
\begin{proof}
Since $T$ is an $A$-isometry operator, then clearly we have $\|T\|_A=1$. Therefore, we get the desired result by applying Theorem \ref{doublecomm}
\end{proof}

We end this paper by the following theorem.
\begin{theorem}
Let $T\in\mathcal{B}_A(\mathcal{H})$ be an $A$-normal operator such that $TS=ST$. Then,
 \begin{equation}\label{hook02000}
\omega_A(TS)\leq \omega_A(T)\omega_A(S).
 \end{equation}
\end{theorem}
\begin{proof}
Let $T\in \mathcal{B}_A(\mathcal{H})$ is an $A$-normal. It is not difficult to see that $\widetilde{T}\in \mathcal{B}(\mathbf{R}(A^{1/2}))$ is a normal operator. So, since $T$ is an $A$-normal operator and satisfies $TS=ST$, then $\widetilde{T}$ is a normal operator on $\mathbf{R}(A^{1/2})$ and satisfies $\widetilde{T}\widetilde{S}=\widetilde{S}\widetilde{T}$. Therefore, by Feglede's theorem \cite{fuglede} we deduce that $\widetilde{T}^*\widetilde{S}=\widetilde{S}\widetilde{T}^*$. This implies, by taking adjoints, that $\widetilde{T}\widetilde{S}^*=\widetilde{S}^*\widetilde{T}$. Hence, by applying \cite[Theorem 1.10]{D.1} we get
$$\omega(\widetilde{T}\widetilde{S})\leq \omega(\widetilde{S})\|\widetilde{T}\|_{\mathcal{B}(\mathbf{R}(A^{1/2}))}.$$
On the other hand, since $\widetilde{T}$ is a normal operator in $\mathcal{B}(\mathbf{R}(A^{1/2}))$, then  $\omega(\widetilde{T})=\|\widetilde{T}\|_{\mathcal{B}(\mathbf{R}(A^{1/2}))}$ (see \cite{bakfeki02}). So, we deduce that
$$\omega(\widetilde{T}\widetilde{S})\leq \omega(\widetilde{S})\omega(\widetilde{T}).$$
Hence, by using \eqref{prosum} and Lemma \ref{lr2} (b), we obtain
$$\omega_A(TS)\leq \omega_A(S)\omega_A(S),$$
as required.
\end{proof}


\end{document}

\documentclass[12pt,reqno]{amsart}
\usepackage{etoolbox}

   \makeatletter

 \patchcmd{\@setaddresses}{\scshape\ignorespaces}{\ignorespaces}{}{} 

\appto\maketitle{%
\let\@makefnmark\relax  \let\@thefnmark\relax
\ifx\@empty\addresses\else\@footnotetext{%
  \vskip-\bigskipamount\@setaddresses}
  }
\def\enddoc@text{}
\makeatother

\makeatletter
\patchcmd\maketitle
  {\uppercasenonmath\shorttitle}
  {}
  {}{}
\patchcmd\maketitle
  {\@nx\MakeUppercase{\the\toks@}}
  {\the\toks@}
  {}
  {}{}
\patchcmd\@settitle{\uppercasenonmath\@title}{\Large}{}{}
\patchcmd\@setauthors
  {\MakeUppercase{\authors}}
  {\authors}
  {}{}
\makeatother
\usepackage{amsmath,amssymb,amsthm}
\usepackage{color}
\usepackage{url}
\usepackage{tikz-cd}
\usepackage[utf8]{inputenc}
\usepackage[T1]{fontenc}
\textheight 22.5truecm \textwidth 14.5truecm
\setlength{\oddsidemargin}{0.35in}\setlength{\evensidemargin}{0.35in}

\setlength{\topmargin}{-.5cm}
\newtheorem{theorem}{Theorem}[section]
\newtheorem{definition}{Definition}[section]
\newtheorem{definitions}{Definitions}[section]
\newtheorem{notation}{Notation}[section]
\newtheorem{corollary}{Corollary}[section]
\newtheorem{proposition}{Proposition}[section]
\newtheorem{lemma}{Lemma}[section]
\newtheorem{remark}{Remark}[section]
\newtheorem{example}{Example}[section]
\numberwithin{equation}{section}
\usepackage[colorlinks=true,pagebackref=true,pagebackref=true]{hyperref}
\hypersetup{urlcolor=blue, citecolor=red , linkcolor= blue}
\usepackage[capitalise,noabbrev,nameinlink]{cleveref}      
\begin{document}
\address{$^{[1]}$ University of Sfax, Sfax, Tunisia.}
\email{\url{kais.feki@hotmail.com}}


\date{\today}
\author[Kais Feki] {\Large{Kais Feki}$^{1}$}
\title[Some numerical radius inequalities for semi-Hilbert space operators]{Some numerical radius inequalities for semi-Hilbert space operators}

\maketitle